\documentclass[12pt]{amsart}
\usepackage[latin1]{inputenc}
\usepackage{color}
\usepackage{enumerate}
\usepackage{amssymb}
\usepackage[all]{xy}
\usepackage{hyperref}

\newtheorem{theorem}{Theorem}[section]
\newtheorem{lemma}[theorem]{Lemma}

\newtheorem{proposition}[theorem]{Proposition}
\newtheorem{corollary}[theorem]{Corollary}
\newtheorem{fact}[theorem]{Fact}
\theoremstyle{definition}

\newtheorem{example}[theorem]{Example}

\theoremstyle{remark}

\newtheorem{question}[theorem]{Question}

\newcommand{\set}[1]{\left\{#1\right\}}

\newcommand{\To}{\longrightarrow}

\newcommand{\U}{\mathcal{U}}
\newcommand{\nat}{\mathbb{N}}
\newcommand{\iten}{\ensuremath{\widehat{\otimes}_\varepsilon}}
\newcommand{\pten}{\ensuremath{\widehat{\otimes}_\pi}}
\newcommand{\R}{\mathbb{R}}
\newcommand{\N}{\mathbb{N}}

\newcommand{\F}{\mathcal{F}}

\newcommand{\I}{\mathcal{I}}

\def\natu{{\mathbb N}}

\DeclareMathOperator{\co}{co}

\newcommand{\ext}[1]{\operatorname{ext}\left(#1\right)}

\DeclareMathOperator{\supp}{supp}

\usepackage{color}
\usepackage{enumerate}
\begin{document}
\setcounter{tocdepth}{1}
%\date{\today}

%\renewcommand{\baselinestretch}{1.1}

\title{$L$-orthogonal elements and $L$-orthogonal sequences}
\author[Avil\'es]{Antonio Avil\'es}
\address[Avil\'es]{Universidad de Murcia, Departamento de Matem\'{a}ticas, Campus de Espinardo 30100 Murcia, Spain
	\newline
	\href{https://orcid.org/0000-0003-0291-3113}{ORCID: \texttt{0000-0003-0291-3113} } }
\email{\texttt{avileslo@um.es}}

\author[Mart\'inez-Cervantes]{Gonzalo Mart\'inez-Cervantes}
\address[Mart\'inez-Cervantes]{Universidad de Murcia, Departamento de Matem\'{a}ticas, Campus de Espinardo 30100 Murcia, Spain
	\newline
	\href{http://orcid.org/0000-0002-5927-5215}{ORCID: \texttt{0000-0002-5927-5215} } }	

\email{gonzalo.martinez2@um.es}

%\author[Rodr\'iguez]{Jos\'e Rodr\'iguez}
%\address[Rodr\'iguez]{Universidad de Murcia, Departamento de Ingenier\'ia y Tecnolog\'ia de Computadores , Campus de Espinardo 30100 Murcia, Spain}
%
%\email{joserr@um.es}

%\author[Rodr\'iguez Abell\'an]{Jos\'e David Rodr\'iguez Abell\'an}
%\address[Rodr\'iguez Abell\'an]{Universidad de Murcia, Departamento de Matem\'{a}ticas, Campus de Espinardo 30100 Murcia, Spain 	\newline
%	\href{https://orcid.org/0000-0002-2764-0070}{ORCID: \texttt{0000-0002-2764-0070} }}
%
%\email{josedavid.rodriguez@um.es}

\author[Rueda Zoca]{Abraham Rueda Zoca}
\address[Rueda Zoca]{Universidad de Murcia, Departamento de Matem\'{a}ticas, Campus de Espinardo 30100 Murcia, Spain
	\newline
	\href{https://orcid.org/0000-0003-0718-1353}{ORCID: \texttt{0000-0003-0718-1353} }}
\email{\texttt{abraham.rueda@um.es}}
\urladdr{\url{https://arzenglish.wordpress.com}}

\thanks{The three authors have been supported by MTM2017-86182-P (Government of Spain, AEI/FEDER, EU) and Fundaci\'on S\'eneca, ACyT Regi\'on de Murcia grant 20797/PI/18. G. Mart\'{\i}nez-Cervantes has been co-financed by the European Social Fund (ESF) and the Youth European Initiative (YEI) under the Spanish S\'eneca Foundation (CARM) (ref. 21319/PDGI/19). A. Rueda Zoca was supported by Juan de la Cierva-Formaci\'on fellowship FJC2019-039973, by MICINN (Spain) Grant PGC2018-093794-B-I00 (MCIU, AEI, FEDER, UE), and by Junta de Andaluc\'ia Grants A-FQM-484-UGR18 and FQM-0185.}

\keywords{$L$-orthogonal element; $L$-orthogonal sequence; $\ell_1$-type sequence; octahedral norm; selective ultrafilter; Ramsey ultrafilter; $Q$-point}

\subjclass[2020]{Primary 46B20, 46B04; Secondary 46B26, 54A20}

\begin{abstract} 
Given a Banach space $X$, we say that a sequence $\{x_n\}$ in the unit ball of $X$ is $L$-orthogonal if $\Vert x+x_n\Vert\rightarrow 1+\Vert x\Vert$ for every $x\in X$. On the other hand, an element $x^{**}$ in the bidual sphere is said to be $L$-orthogonal (to $X$) if $\|x+x^{**}\|= 1+\Vert x\Vert$ for every $x\in X$. 	
A result of V. Kadets, V. Shepelska and D. Werner asserts that a Banach space contains an isomorphic copy of $\ell_1$ if and only if there exists an equivalent renorming with an $L$-orthogonal sequence, whereas a result of G. Godefroy claims that containing an isomorphic copy of $\ell_1$ is equivalent to the existence of an equivalent renorming with $L$-orthogonals in the bidual.
The aim of this paper is to clarify the relation between $L$-orthogonal sequences and $L$-orthogonal elements. Namely,	
we study whether every $L$-orthogonal sequence contains $L$-orthogonal elements in its weak*-closure. We provide an affirmative answer whenever the ambient space has small density character. Nevertheless, we show that, surprisingly, the general answer is independent of the usual axioms of set theory.
We also prove that, even though the set of $L$-orthogonals is not a vector space, this set contains infinite-dimensional Banach spaces when the surrounding space is separable.
%
%In this section we analyse the problem whether the existence of a sequence $\{x_n\}$ in the unit ball so that $\Vert x+x_n\Vert\rightarrow 1+\Vert x\Vert$ for every $x\in X$ implies the existence of $u$ in the $w^*$-closure of $\{x_n\}$ so that $\Vert x+u\Vert=1+\Vert x\Vert$ holds for every $x\in X$. We prove that, for low density characters, the answer is always affirmative. Moreover, we show that, in general, the answer is independent of usual axioms of set theory. We also prove that, even though the set of those $u$ described abover is not a vector space, this set contains infinite-dimensional Banach spaces when the surrounding space is separable.
\end{abstract}

\maketitle

\section{Introduction}

The presence of sequences equivalent to the $\ell_1$ basis and the relation to orthogonality conditions in the corresponding bidual has been a central topic in Banach space theory. For instance, a celebrated characterisation of containment of $\ell_1$ due to B. Maurey reads as follows: a separable Banach space $X$ contains an isomorphic copy of $\ell_1$ if, and only if, there exists a non-zero element $x^{**}\in X^{**}$ so that
$$\Vert x+x^{**}\Vert=\Vert x-x^{**}\Vert$$
holds for every $x\in X$ \cite{maurey}. Soon afterwards, several related conditions appeared at the end of the eighties by a series of works of G. Godefroy, N. J. Kalton and B. Maurey. For instance, it was proved in \cite[Theorem II.4]{god} that a Banach space $X$ contains an isomorphic copy of $\ell_1$ if, and only if, $X$ admits an equivalent renorming so that there are non-zero elements $x^{**}\in X^{**}$ satisfying
\begin{equation}\label{ecua:lortointro}
\Vert x+x^{**}\Vert=1+\Vert x\Vert
\end{equation}
holds for every $x\in X$ (a version for separable Banach spaces was obtained previously by G. Godefroy and B. Maurey, see the comment preceding \cite[Theorem II.4]{god}). For simplicity, and according to the notation in \cite{loru}, an element $x^{**}$ satisfying \eqref{ecua:lortointro} will be said to be an \textit{$L$-orthogonal element}.

This condition was also studied, in contrast to the previous result, from an isometric point of view. In \cite[Lemma 9.1]{gk} it was proved that a separable Banach space $X$ admits an $L$-orthogonal element if and only if the norm of $X$ is \textit{octahedral}, i.e.~if given any finite-dimensional subspace $F$ of $X$ and every $\varepsilon>0$ there is always an element $x\in S_X$ so that $\Vert y+\lambda x\Vert\geq (1-\varepsilon)(\Vert y\Vert+\vert\lambda\vert)$ holds for every $y\in F$ and every $\lambda\in\mathbb R$. 

The strong connection between octahedral norms and other properties of Banach spaces like the Daugavet property \cite[Lemma 2.8]{kssw} and the diameter two properties \cite[Lemma 2.1]{blrjfa} has motivated further research on octahedral norms (see e.g. \cite{hlp,llr2,lr,pr}). However, the question whether octahedrality implies the existence of non-zero elements enjoying \eqref{ecua:lortointro} has remained open until the recent work \cite{loru}, where an octahedral Banach space of density character $2^\mathfrak{c}$ is constructed for which no $L$-orthogonal elements exist.

In this paper we address, inspirated by the work with sequences done in \cite[Lemma 1.4]{maurey}, the question whether there are $L$-orthogonal elements in Banach spaces in which the octahedrality condition comes from a unique $\ell_1$ sequence. To be more precise, given a Banach space $X$, we say that a sequence $\{x_n\}\subseteq B_X$ is an \textit{$L$-orthogonal sequence} if $\Vert x+x_n\Vert\rightarrow 1+\Vert x\Vert$ holds for every $x\in X$. The connection between this concept and Maurey's work \cite{maurey} is that, if one defines the function $\tau(x):=\lim_n \Vert x+x_n\Vert$, then $\tau$ defines a type of Maurey satisfying $\tau(x)=1+\Vert x\Vert$ for every $x\in X$. For this reason, $L$-orthogonal sequences receive in \cite{ksw} the name \textit{$\ell_1$-type}. In that paper, V. Kadets, V. Shepelska and D. Werner prove  that every Banach space containing an isomorphic copy of $\ell_1$ admits an equivalent renorming with $L$-orthogonal sequences \cite[Theorem 4.3]{ksw}.

If the norm of a separable Banach space $X$ is octahedral then there exists an $L$-orthogonal sequence in $X$, see e.g.~\cite[Theorem 1.1]{ksw}. It is easy to give examples of nonseparable spaces that have $L$-orthogonal elements but with no $L$-orthogonal sequences. We focus on the converse implication.

\begin{question}\label{question:general}
Let $X$ be a Banach space with an $L$-orthogonal sequence $\{x_n\}$. Is there an $L$-orthogonal element $u$ in $X^{\ast\ast}$?
\end{question}

The natural place to look for this $u$ is within the weak$^\ast$ cluster points of the sequence, so we formulate a second question:

\begin{question}\label{question:main}
Let $X$ be a Banach space with an $L$-orthogonal sequence $\{x_n\}$. Is there an $L$-orthogonal element $u$ in the $w^*$-closure of $\{x_n\}$?
\end{question}

The closure of $\{x_n\}$ is the set of limits of the sequence under ultrafilters, so what we are asking is whether the $w^*$-limit of an $L$-orthogonal sequence under an ultrafilter will be an $L$-orthogonal element. 

We present two different positive results which, with different approaches, take advantage of Maurey's techniques in \cite{maurey} together with some set theory, allowing us to select appropiate $w^*$-cluster points to obtain $L$-orthogonals. In Section \ref{sect:cardinalp} we prove that if $X$ has an $L$-orthogonal sequence $\{x_n\}$ and the density character of $X$ does not exceed the pseudointersection number $\mathfrak{p}$, then there exists an $L$-orthogonal $w^*$-cluster point of $\{x_n\}$. Our second result removes any restriction on the density of $X$ at the expense of using devices that require extra axioms. In Section \ref{sect:selective} we show that the $w^\ast$-limit of $\{x_n\}$ through a selective ultrafilter is an $L$-orthogonal element of the bidual. Selective ultrafilters exist in some models of set theory, for instance under the Continuum Hypothesis. It follows that a general positive answer to Question \ref{question:main}, and hence also to Question \ref{question:general}, is consistent with the usual axioms of set theory. This gives, as a particular case, a consistent answer to a question posed in \cite{ksw}.

In Section \ref{section:counterexamples} we turn to the negative side. First we exhibit an $L$-orthogonal sequence $\{x_n\}$ in $\ell_1$ such that $0$ is in the weak closure of $\{x_n\}$. This reveals that some selection of cluster points is necessary even in the separable case to obtain $L$-orthogonals in the $w^*$-closure of an $L$-orthogonal sequence. In the second place, we prove that a negative answer to Question \ref{question:main} is also consistent: In some models of set theory there are $L$-orthogonal sequences without $L$-orthogonal $w^\ast$-cluster points. What we show is that there is an $L$-orthogonal sequence for which none of whose limits through ultrafilters that are not $Q$-points are ever $L$-orthogonal. Moreover, notice that there are models of set theory where no ultrafilter is a $Q$-point.

The conclusion is that Question \ref{question:main} is undecidable in ZFC. For Question \ref{question:general} we give a consistent positive answer, but we leave open whether a negative answer may also be consistent. Another issue is that we consider selective ultrafilters for the positive result, but the larger class of $Q$-points for the negative one. When taking limits through ultrafilters that are $Q$-points but not selective, we do not know if we always get $L$-orthogonal elements. 

In the final Section \ref{section:abundLorto} we analyse how large the set of $L$-orthogonal elements may be, when it is nonempty. Making use of tensor product spaces, we prove that if $X$ is a separable Banach space with octahedral norm, then the set of multiple scalars of $L$-orthogonal elements in $X^{**}$ contains infinite-dimensional spaces, actually an isometric copy of the dual of any separable and uniformly convex Banach space with the metric approximation property which is finitely representable in $\ell_1$ (Theorem \ref{theo:espaciaorto}).

The content of other sections in the article is the following. In Section \ref{sect:examples} we give some basic natural examples of $L$-orthogonal sequences and $L$-orthogonal elements and some cases when they do not exist. Section \ref{sect:maurey} recaps some technical material from \cite{maurey} that will be of key use in some proofs. In Section \ref{section:Lebesgue} we prove a Lebesgue dominated convergence theorem for convergence with respect to the ideal of all subsets of $\mathbb{N}$ that have uniformly bounded intersections with a partition into finite sets. This will be auxiliary in Section \ref{section:counterexamples} but may have an independent interest in the line of the study of Lebesgue filters \cite{KadetsLeonov,KadetsLeonov2}. 

We will consider only real Banach spaces. Let us fix some general notation: Given a Banach space $X$, $B_X = \{x\in X : \|x\|\leq 1\}$ and $S_X=\{x\in X : \|x\|=1\}$ will be the closed unit ball and the unit sphere of $X$, respectively. Given a set $\Omega$, $\mathcal{P}(\Omega) = \{A : A\subseteq \Omega\}$ will be the power set of $\Omega$, and for $A\subset \Omega$, $\chi_A$ will be the characteristic function of $A$, so that $\chi_A(x)=1$ if $x\in A$ and $\chi_A(x)=0$ if $x\in\Omega\setminus A$. By a measure we mean a \textit{regular Borel signed measure}.

\section{Examples of $L$-orthogonal sequences and $L$-orthogonal elements}\label{sect:examples}

The most canonical example of an $L$-orthogonal sequence is the unit basis of $\ell_1$, or more generally any sequence of different vectors from the canonical basis of $\ell_1(\Gamma)$. In this case $\ell_1(\Gamma)^{**}$ is identified with the space of finitely additive signed measures defined on  $\mathcal{P}(\Gamma)$, and the $L$-orthogonal elements in the sphere of $\ell_1(\Gamma)^{**}$ correspond to the measures that vanish on all finite sets. The $w^*$-cluster points of the canonical basis, identified with $\{0,1\}$-valued measures given by ultrafilters, are examples of such $L$-orthogonal vectors.

In a space $L_1(\mu)$ with $\mu$ non-atomic, an $L$-orthogonal sequence can be constructed as  $\{{\mu(A_n)}^{-1}\chi_{A_n}\}$ where  $\{A_n\}$ is a decreasing sequence of nonzero sets whose measures converge to 0. The bidual space $L_1(\mu)^{**}$ is identified with the space of finitely additive signed measures on the measure algebra $\mathcal{M}_\mu$, which in turn is identified with the space of regular Borel signed measures on its Stone space $St(\mathcal{M}_\mu)$. Being an $L$-orthogonal element corresponds to being a signed measure of total variation 1 that is orthogonal to $\mu$.  Any $w^*$-cluster point of the above sequence will correspond to a probability measure $\nu$ that will be concentrated on $A_n$ for all $n$, and it will be again an $L$-orthogonal element.

In a $c_0$-sum of uncountably many copies of $\ell_1$ there are no $L$-orthogonal sequences $(x_n)$ because, as soon as we take a vector $x$ whose support is disjoint from the supports of all $x_n$ we would get $\|x+x_n\| = \max\{\|x\|,\|x_n\|\}$. However, the bidual space would be an $\ell_\infty$-sum of copies of $\ell_1^{**}$ and if we take an $L$-orthogonal element for $\ell_1$ in each coordinate, we obtain an $L$-orthgonal element.

In the case of $C(K)$ spaces, playing with the signs is necessary. The Rademacher functions on $C[0,1]$, or a sequence of coordinate functions on $C(\{-1,1\}^\Gamma)$ are the most natural examples of $L$-orthogonal sequences in this context, and the $w^*$-cluster points of such sequences give $L$-orthogonal elements.

\begin{proposition}
	The space $C(K)$ has an $L$-orthogonal element if and only if $K$ has no isolated points.
\end{proposition} 

\begin{proof}
	Suppose that $K$ has an isolated point $p$ and $x^{**}$ was an $L$-orthogonal element. Fix $\varepsilon>0$ and take $\nu\in B_{C(K)^{*}}$ such that $(x^{**} + \chi_{\{x\}})(\nu) > 2-\varepsilon$. Call $t= \nu\{x\}$. We have that $$2-\varepsilon <(x^{**} + \chi_{\{x\}})(\nu) = x^{**}(\nu-t\delta_x) + x^{**}(t\delta_x) + t.$$  
	But $\|\nu-t\delta_x\| = \|\nu\|-|t| \leq 1-t$, so  $x^{**}(t\delta_x)>1-\varepsilon$. Since $|t|\leq 1$ and this holds for every $\varepsilon>0$, the conclusion is that $|x^{**}(\delta_x)|=1$. In fact, $x^{**}(\delta_x)=1$, because otherwise the chain of inequalities displayed above would give $2-\varepsilon < x^{**}(\nu-t\delta_x) \leq 1$. We can repeat the computation using the fact that $(x^{**}-\chi_{\{x\}})(\nu)>2-\varepsilon$ for some $\nu\in B_{C(K)^*}$, and we obtain a contradiction.
	We pass to the converse implication. By Riesz's theorem, every element of $C(K)^\ast$ is given by integration with respect to a measure. Every such measure $\nu$ has a canonical decomposition $\nu = \nu_a + \nu_n$ into its atomic and its atomless part. The formula $x^{**}(\nu) = \nu_a(K) - \nu_n(K)$ defines an $L$-orthogonal element in $C(K)^{**}$. It is clear that $\|x^{**}\|=1$. Take $f\in C(K)$ and fix a point $p$ where $|f|$ attains its maximum. If $f(p) = \|f\|$ then $\|x^{**}+f\| \geq (x^{**}+f)(\delta_p) = 1 +\|f\|$. If $f(p) = -\|f\|$, take any $\varepsilon>0$, a neighborhood $W$ of $p$ where $f(t)<\varepsilon-\|f\|$ holds for every $t\in W$. Since $K$ has no isolated points, there exists an atomless probability measure $\nu$ supported inside $W$, cf. \cite[Theorem 19.7.6]{semadeni}. We get $\|x^{**} + f\| \geq (x^{**}+f)(-\nu) = 1 + \|f\| -\varepsilon$.
\end{proof}

\begin{proposition}\label{CKorthogonalseq}
	The space $C(K)$ has an $L$-orthogonal sequence if and only if there exist two sequences $\{C_n\},\{D_n\}$ of closed sets with $C_n\cap D_n=\emptyset$ and so that, for every non-empty open subset $O$ in $K$, there exists $m\in\mathbb N$ so that $C_n\cap O\neq \emptyset$ and $D_n\cap O\neq \emptyset$ holds for every $n\geq m$.
\end{proposition}

\begin{proof}
	If $\{f_n\}$ is an $L$-orthogonal sequence, consider the sets $C_n=\{x\in K : f_n(x)\geq 1/2\}$ and $D_n=\{x\in K : f_n(x)\leq -1/2\}$. Given a nonempty open set $O$ we can find a continuous $f:K\To [0,1]$ that vanishes on $K\setminus O$ and takes value 1 on some point of $O$. We just apply the definition of $L$-orthogonal sequence to $f$ to see that $O\cap C_n$ and $O\cap D_n$ are eventually nonempty. For the converse implication, for every $n\in\mathbb{N}$, by Tietze's extension theorem, find a continuous function $f_n:K\To [-1,1]$ that is constant equal to $1$ on $C_n$ and constant equal to $-1$ on $D_n$. We check that $\{f_n\}$ is an orthogonal sequence. To this end, pick $f\in C(K)$ and $\varepsilon>0$. Assume that there exists $t\in K$ so that $f(t)>\|f\|-\varepsilon$ (the case $f(t)<-\|f\|+\varepsilon$ runs similarly) and define the open set $O:=\{t\in K: f(t) >\|f\|-\varepsilon\}$. Find $m\in\mathbb N$ such that $O\cap C_n\neq \emptyset$ and $O\cap D_n\neq \emptyset$ for all $n\geq m$. Given one such $n$ we get, evaluating at a point of $C_n\cap O$, $$1+\|f\|-\varepsilon<\Vert f+f_n\Vert \leq \|f\| + \|f_n\| = \|f\|+1.$$
\end{proof}

Examples of compact spaces $K$ that satisfy the hypotheses of Proposition \ref{CKorthogonalseq} are metric spaces without isolated points (the sets $C_n$ and $D_n$ can be taken as finite sets that meet every ball of a finite cover with balls of radius $1/n$) and any infinite product of nontrivial compact spaces  (taking $C_n$ and $D_n$ the elements whose $i_n$-th coordinate belongs to two disjoint closed sets). An easy example of $C(K)$ where $L$-orthogonal elements exist but $L$-orthogonal sequences do not is the one-point Alexandroff compactification of the disjoint union of uncountably many compact spaces without isolated points. For some $n$, $C_n$ and $D_n$ should both intersect infinitely many sets in that union, and that would oblige the infinity point to belong to $C_n\cap D_n$. Another such example is the \v{C}ech-Stone remainder of the natural numbers, $K=\beta\mathbb{N}\setminus\mathbb{N}$. If we had such $C_n$ and $D_n$ there, and if we take $S_n$ a clopen set in $\beta \N$ separating $C_n$ and $D_n$, then $\{S_n\cap \N : n\in\mathbb{N}\}$ would give what is called a \emph{splitting family}, and all those families are uncountable, cf. \cite[p. 182]{halbeisen}. Thus $\ell_\infty/c_0 = C(\beta\mathbb{N}\setminus\mathbb{N})$ fails to have an $L$-orthogonal sequence, though it has an $L$-orthogonal element since $\beta\mathbb{N}\setminus\mathbb{N}$ has no isolated point. The space $\ell_\infty=C(\beta\mathbb{N})$ lacks both.

Another example that we may look is the Lipschitz-free space $\mathcal{F}(M)$ over a metric space $M$. If we have bounded sequences $(u_n)_n,(v_n)_n$ in $ M$ with $u_n\neq v_n$ and $\lim_n d(u_n,v_n)=0$, then $\{\frac{\delta_{u_n}-\delta_{v_n}}{d(u_n,v_n)}\}$ is an $L$-orthogonal sequence in $\mathcal F(M)$ by \cite[Theorem 2.6]{jr}. Note that every $w^*$-cluster point of $\{\frac{\delta_{u_n}-\delta_{v_n}}{d(u_n,v_n)}\}$ produces an $L$-orthogonal element \cite[Lemma 2.11]{ap}.

Finally, let us mention that an analogue class to octahedral spaces called \emph{almost square spaces} has been studied in \cite{aht,abrahamsenlangemetslima}, where the role of the $\ell_1$-norm is played by the $\ell_\infty$-norm instead. If a Banach space is almost square, then its dual $Y^*$ is octahedral \cite[Proposition 2.5]{abrahamsenlangemetslima}. A natural variation in our setting would be the following: We define a Banach space $X$ to be \emph{sequentially almost square} if there exists a sequence $(y_n)\subseteq S_Y$ such that $\Vert y+y_n\Vert\rightarrow \|y\|$ holds for every $y\in Y$. Though not formally defined in the literature, this concept has been implicitly used for instance, in \cite[Proposition 2.3]{gr} and in \cite[Lemma 4.4]{aht}.

\begin{proposition}\label{prop:ejesecuorto}
If a Banach space $Y$ is sequentially almost square, then $Y^*$ has an $L$-orthogonal sequence.
\end{proposition}

\begin{proof}
Let $(y_n)$ be a sequence in $S_Y$ witnessing that $Y$ is sequentially almost square. For every $n\in\mathbb N$ take, by Hahn-Banach Theorem, an element $y^*_n\in S_{Y^*}$ so that $y^*_n(y_n)=1$. Let us prove that $(y^*_n)$ is an $L$-orthogonal sequence. To this end, pick $y^*\in S_{Y^*}$ and $\varepsilon>0$. Find $y\in S_Y$ so that $y^*(y)>1-\varepsilon$. By assumption, there exists $m\in\mathbb N$ so that $\Vert y\pm y_n\Vert<1+\varepsilon$ holds for every $n\geq m$. Now, given $n\geq m$, note that 
$$1+\varepsilon > y^*(y\pm y_n)=y^*(y)\pm y^*(y_n)>1-\varepsilon\pm y^*(y_n),$$
so making a suitable choice of sign we get $1-\varepsilon+\vert y^*(y_n)\vert<1+\varepsilon$, from where $\vert y^*(y_n)\vert<2\varepsilon$. In a similar way, since $y^*_n(y_n)=1$, we obtain that $\vert y^*_n(y)\vert<\varepsilon$. Hence
$$\Vert y^*+y^*_n\Vert\geq \frac{(y^*+y^*_n)(y+y_n)}{\Vert y+y_n\Vert}>\frac{2-3\varepsilon}{\Vert y+y_n\Vert}>\frac{2-3\varepsilon}{1+\varepsilon},$$
since $n>m$ was arbitrary we conclude the desired result by Lemma \ref{lemma:abal}.
\end{proof}

All separable almost square spaces are sequentially almost square (see the proof of \cite[Proposition 2.3]{gr}). For non-separable examples, it follows from the results of \cite[Section 3]{gr} that if $M$ is a locally compact, totally disconnected and not uniformly discrete metric space $M$, then the space of little Lipschitz functions $lip(M)$ is sequentially almost square. Also, if $X$ is sequentially almost square and $Y$ is a non-zero Banach space, then the space of compact operators $K(Y,X)$ is sequentally almost square, arguing like in \cite[Proposition 2.6]{gr}.

Let us finish with a well known result which is is very useful when dealing with $L$-orthogonal elements and which will be used without any explicit mention. A proof can be found, for instance, in \cite[Lemma 11.4]{abal}.

\begin{lemma}\label{lemma:abal}
Let $X$ be a Banach space and let $u,v\in X$ be two vectors so that
$$\Vert u+v\Vert=\Vert u\Vert+\Vert v\Vert.$$
Then the equality $\Vert \alpha u+\beta v\Vert=\alpha\Vert u\Vert+\beta \Vert v\Vert$ holds for every $\alpha,\beta\geq 0$.
\end{lemma}

\section{Maurey types techniques}\label{sect:maurey}

The goal of this section is to prove Lemmata \ref{lemma:previmaure} and \ref{lemma:sepamaurey}, which will play a fundamental role in Sections \ref{sect:cardinalp} and \ref{sect:selective}. Although these results are in essence a simple reformulation of \cite[Remark 1.3 and Lemma 1.4]{maurey}, we include proofs which we think might be more transparent for the reader.
First, we need the following auxiliary lemma.

\begin{lemma}\label{lemma:auxiliar} Let $X$ be a Banach space with an $L$-orthogonal sequence $\{x_n\}$. For every finite-dimensional subspace $F$ of $X$ and every $\varepsilon>0$ there exists $m\in\mathbb N$ so that
	$$\Vert y+\lambda x_n\Vert>(1-\varepsilon)(\Vert y\Vert+\vert \lambda\vert)$$
	holds for every $y\in F$, every $\lambda\in\mathbb R$ and every $n\geq m$.
\end{lemma}

\begin{proof} 
	Fix $F$ and $\varepsilon$. Since $S_F$ is compact, we can find $\{y_1,\ldots,y_k\}\subseteq S_F$ such that $S_F\subseteq \bigcup\limits_{i=1}^k B\left( y_i,\frac{\varepsilon}{2}\right)$. Since $\{x_n\}$ is an $L$-orthogonal sequence, find $m\in\mathbb N$ so that $n\geq m$ implies
	$$\Vert y_i+x_n\Vert\geq 2-\frac{\varepsilon}{2}\ \mbox{ for every } i\in\{1,\ldots, k\}.$$
	Fix $n\geq m$. Since $S_F\subseteq \bigcup\limits_{i=1}^n B\left( y_i,\frac{\varepsilon}{2}\right)$, the triangle inequality implies that
	$$\Vert y+x_n\Vert\geq 2-\varepsilon$$
	holds for all $y\in S_F$. Let us conclude from here the desired inequality.
	Let $y\in S_F$ and $t_1,t_2 \in \R$ be arbitrary. It is enough to show that 
	$$\|t_1x_n+t_2y\| > (1-\varepsilon)(|t_1|+|t_2|)$$
	holds for every $n\geq m$.
	Dividing by a suitable constant, we can suppose that $|t_1|+|t_2|=1$. Moreover, taking $-y$ instead of $y$ if necessary, we can suppose that $t_1$ and $t_2$ are both positive.
	We assume that $t_1\geq t_2$ (the other case is analogous). Then  
	\[\begin{split}\Vert t_1 x_n+t_2 y\Vert& =\Vert t_1(x_n+y)+(t_2-t_1)y\Vert\\
	& \geq t_1\Vert x_n+y\Vert-\vert t_2-t_1\vert\Vert y\Vert\\
	& \geq t_1(2-\varepsilon)+t_2-t_1=t_1+t_2-t_1\varepsilon\\
	& \geq 1-\varepsilon\end{split}\]
	as desired.
\end{proof}

\begin{lemma}\label{lemma:previmaure}
Let $X$ be a Banach space with an $L$-orthogonal sequence $\{x_n\}$. Take $Z$ a separable subspace of $X$ and $\varepsilon_n$ a sequence of positive real numbers. Write $Z=\overline{\bigcup\limits_{n\in\mathbb N} F_n}$, where $\{F_n\}$ is an increasing sequence of finite dimensional subspaces of $X$. Then there exists a subsequence $\{y_n\}$ of $\{x_n\}$ so that 
$$(1-\varepsilon_n)(1+\Vert x\Vert)\leq \left\Vert x+\sum_{j=n+1}^\infty \alpha_j y_j\right\Vert$$
holds for every $n\in\mathbb N$, $x\in F_n$ and every $\{\alpha_j\}\subseteq \mathbb R^+$ with $\sum_{j=n+1}^\infty \alpha_j=1$.
\end{lemma}

\begin{proof}
	Pick $(\delta_n)$ a sequence of positive scalars so that $(1-\varepsilon_n)<\prod_{j=n+1}^\infty (1-\delta_j)$. By induction, we can construct inductively a subsequence $(y_n)$ of $(x_n)$ so that
	$$\Vert y+\lambda y_n\Vert\geq (1-\delta_n)(\Vert y\Vert+\vert\lambda\vert)$$
	holds for every $y\in span(F_n\cup\{y_1,\ldots, y_{n-1}\})$ and every $\lambda\in\mathbb R$. Let us prove that the sequence $(y_n)$ satisfies our requirements. To this end, pick $n\in\mathbb N$, $x\in F_n$ and ${\alpha_j}$ a sequence of positive scalars with $\sum_{j=n+1}^\infty \alpha_j=1$. Then, for every $k>n$ it follows that
	\[\begin{split}
	\left\Vert x+\sum_{j=n+1}^k \alpha_j y_j\right\Vert& \geq (1-\delta_k)\left(\left\Vert x+\sum_{j=n+1}^{k-1}\alpha_j y_j\right\Vert+\alpha_k\right)\\
	& \geq (1-\delta_k)\left((1-\delta_{k-1})\left(\left\Vert x+\sum_{j=n+1}^{k-2}\alpha_j y_j\right\Vert+\alpha_{k-1}\right)+\alpha_k\right)\\
	& \geq (1-\delta_k)(1-\delta_{k-1})\left(\left\Vert x+\sum_{j=n+1}^{k-2}\alpha_j y_j\right\Vert+\alpha_{k-1}+\alpha_k\right).
	\end{split} 
	\]
	Continuing in this fashion we get that 
	\[\begin{split}
	\left\Vert x+\sum_{j=n+1}^k \alpha_j y_j\right\Vert& \geq \prod\limits_{i=n+1}^k(1-\delta_i)\left(\Vert x\Vert+\sum_{j=n+1}^k\alpha_j\right)\\
	& \geq \prod\limits_{i=n+1}^\infty(1-\delta_i)\left(\Vert x\Vert+\sum_{j=n+1}^k\alpha_j\right)\\
	& \geq (1-\varepsilon_n)\left(\Vert x\Vert+\sum_{j=n+1}^k\alpha_j\right).
	\end{split} 
	\]
	Note that $(x+\sum_{j=n+1}^k \alpha_j y_j)_k$ converges in norm to $x+\sum_{j=n+1}^\infty \alpha_j y_j$. Since the previous inequality holds for every $n\in\mathbb N$ we get that
	$$\left\Vert x+\sum_{j=n+1}^k \alpha_j y_j\right\Vert\geq (1-\varepsilon_n)\left(\Vert x\Vert+\sum_{j=n+1}^\infty\alpha_j\right),$$
	as desired.
\end{proof}

The next lemma was implicitly proved by Maurey \cite[Lemma 1.4]{maurey}, though we include a short proof for the sake of completeness.

\begin{lemma}\label{lemma:sepamaurey}
Let $X$ be a Banach space with an $L$-orthogonal sequence $\{x_n\}\subseteq S_X$. Let $Z$ be a separable subspace of $X$ and let $\{F_n:n\in\mathbb N\}$ be an increasing sequence of finite-dimensional subspaces of $X$ so that $\bigcup\limits_{n\in\mathbb N} F_n$ is dense in $Z$. Pick $\varepsilon_n$ a sequence of positive scalars converging to zero and let $\{y_n\}$ be a subsequence of $\{x_n\}$ satisfying that 
$$(1-\varepsilon_n)(1+\Vert x\Vert)\leq \left\Vert x+\sum_{j=n+1}^\infty \alpha_j y_j\right\Vert$$
holds for every $n\in\mathbb N$, $x\in F_n$ and every $\{\alpha_j\}\subseteq \mathbb R^+$ with $\sum_{j=n+1}^\infty \alpha_j=1$.

Then, any $u\in \bigcap\limits_{n\in\mathbb N} \overline{conv}^{w^*}\{y_j:j>n\} \subseteq B_{X^{**}}$ (in particular, any weak*-cluster point of the sequence) satisfies that
$$\Vert x+u\Vert=1+\Vert x\Vert$$
for every $x\in Z$.
\end{lemma}

\begin{proof}
Given any $x\in Z$ we can assume, up to a density argument, that there exists $m\in\mathbb N$ so that $x\in F_n$ holds for every $n\geq m$. Given $n\geq m$, notice that the condition on the sequence $\{y_n\}$ implies that 
$$x+conv\{y_j:j>n\}\cap (1-2\varepsilon_n)(1+\Vert x\Vert)B_X=\emptyset.$$
By \cite[Corollary 2.13 (b)]{fab} there exists $f_n\in S_{X^*}$ so that $f_n(x+z)\geq (1-2\varepsilon_n)(1+\Vert x\Vert)$ holds for every $z\in conv(\{y_j:j>n\})$. Note that the condition on $f_n$ also implies that 
\begin{equation}\label{eq:s2ortmaurey}
x+\overline{conv\{y_j:j>n\}}^{w^*}\cap (1-2\varepsilon_n)(1+\Vert x\Vert)B_{X^{**}}=\emptyset
\end{equation}
Finally,
if $u\in \bigcap\limits_{n\in\mathbb N}\overline{conv\{y_j:j>n\}}^{w^*}$ then, for every $n>m$, \eqref{eq:s2ortmaurey} implies that $\Vert x+u\Vert\geq (1-2\varepsilon_n)(1+\Vert x\Vert)$.
\end{proof}

Putting everything together what we have is that, when restricting to a separable subspace $Z$, an $L$-orthogonal sequence to $Z$ has a subsequence all of whose block subsequences are $L$-orthogonal to $Z$.

\section{Banach spaces of density at most $\mathfrak{p}$}\label{sect:cardinalp}

We briefly recall what the pseudointersection number $\mathfrak{p}$ is. For two sets $N$, $M$, the notation $N \subseteq^* M$ means that $N \setminus M$ is finite. Given a family of sets $\mathcal{F}$, a pseudointersection for $\mathcal{F}$ is an infinite set $A$ such that $A\subseteq^* B$ for all $B\in\mathcal{F}$. The cardinal $\mathfrak{p}$ is the least cardinality of a family of subsets of $\mathbb{N}$ such that every intersection of elements of $\mathcal{F}$ is infinite, but $\mathcal{F}$ has no pseudointersection. 
A related cardinal $\mathfrak{t}$ has been historically considered, where instead of demanding that every finite intersection is infinite, we require the stronger condition that $\mathcal{F}$ must be a reversely well-ordered family of inifinite sets, in the form $\mathcal{F} = \{N_\alpha : \alpha <\kappa\}$ with $\kappa$ an ordinal and $N_\alpha \subseteq^* N_\beta$ for every  $\beta<\alpha$. It is a simple exercise that these cardinals must be uncountable, so we have
$$ \aleph_1 \leq \mathfrak{p} \leq \mathfrak{t} \leq \mathfrak{c}. $$

We refer the reader to \cite{barto,halbeisen} for technical details and background. The two inequalities at the extremes may become equalities in some models of set theory but strict inequalities in other models. It was shown however in \cite{MalShe}, after being a longstanding open problem, that $\mathfrak{p} = \mathfrak{t}$.

%
%As we have pointed out in the introduction, it is known from \cite[Theorem 3.2]{loru} that there are Banach spaces $X$ whose norm is octahedral but whose bidual space does not contain any $L$-orthogonal element to $X$. From the abundance of spaces with $L$-orthogonal sequences, we will analyse whether, given a (non-separable) Banach space $X$, if $\{x_n\}$ is an orthogonal sequence of $S_X$ we may find an $L$-orthogonal element $u$ being a $w^*$ cluster point of $\{x_n\}$. To this end, we will make use of transfinite techniques to select appropiate cluster points of $\{x_n\}$.

\begin{theorem}\label{theo:octahedral}
Let $X$ be a Banach space with $dens(X) \leq \mathfrak{p}$, and suppose that $\{x_n\}$ is an $L$-orthogonal sequence. Then there is an $L$-orthogonal $w^*$-cluster point of $\{x_n\}$.
\end{theorem}

\begin{proof}
Let $\{Z_\alpha\}_{\alpha < \mathfrak{p}}$ be a family of separable subspaces of $X$  whose union is dense in $X$ and such that the sequence $\{x_n\}_{n \in \natu}$ is contained in each $Z_\alpha$. Fix a sequence $(\varepsilon_n)_n$ of positive real numbers converging to zero and increasing sequences $\{F_n^\alpha:n \in \N\}$ of finite-dimensional subspaces of $X$ such that $\bigcup_{n\in \N}F_n^\alpha$ is dense in $Z_\alpha$ for every $\alpha<\mathfrak{p}$. Our aim will be to construct a chain of subsequences $\{x_n\}_{n \in N_\alpha}$ for every $\alpha<\mathfrak{p}$ such that
\begin{enumerate}
\item[(i)] If $\beta<\alpha$ then $N_\alpha \subseteq^* N_\beta$;
\item[(ii)] $\{x_n\}_{n\in N_\alpha}$ satisfies the hypothesis of Lemma \ref{lemma:sepamaurey} with respect to $Z_\alpha$,  $\{F_n^\alpha:n \in \N\}$ and $(\varepsilon_n)_n$.
\end{enumerate}
The construction will be done by transfinite induction on $\alpha<\mathfrak{p}$. We can take $\{x_n\}_{n \in N_0}$ a subsequence of $\{x_n\}_{n \in \natu}$ like in Lemma \ref{lemma:previmaure} for $Z=Z_0$,  $\{F_n^0:n \in \N\}$ and $(\varepsilon_n)_n$. For the inductive step, assume $\alpha>0$ and that $\{x_n\}_{n \in N_\beta}$ has been constructed for every $\beta<\alpha$. Let us construct $\{x_n\}_{n \in N_\alpha}$. First, since $\alpha < \mathfrak{p}$, there exists an infinite set $M_\alpha$ such that $M_\alpha \subseteq^* N_\beta $ for every $\beta < \alpha$. The subsequence $\{x_n\}_{n \in N_\alpha}$ is obtained by applying Lemma \ref{lemma:previmaure} to the sequence $\{x_n\}_{n \in M_\alpha}$ and the space $Z_\alpha$. This choice guarantees property (ii), while property (i) follows from the fact that $N_\alpha \subseteq M_\alpha$.  This completes the inductive construction.

 For every $\alpha<\mathfrak{p}$ and $n\in\mathbb N$ define $A_{n,\alpha}:=\overline{\{x_j: j\in N_\alpha, j>n\}}^{w^*}\subseteq B_{X^{**}}$. By property (i), the family $\{A_{n,\alpha}: \alpha<\mathfrak{p}, n\in\mathbb N\}$ has the finite intersection property. By compactness, there exists $u$ that belongs to all sets $A_{n,\alpha}$ for $n\in\mathbb N$, $\alpha<\mathfrak{p}$. This $u$ is a $w^*$-cluster point of the sequence $\{x_n\}_{n\in N_\alpha}$ for every $\alpha<\mathfrak{p}$, and in particular, it is a $w^*$-cluster point of $\{x_n\}_{n\in\mathbb N}$. Finally, we prove that $u$ is $L$-orthogonal to $X$. By  Lemma \ref{lemma:sepamaurey}, for any $x\in Z_\alpha$ and any $w^*$-cluster point $v$ of $\{x_n\}_{n\in N_\alpha}$ we have that
$$\Vert x+v\Vert=1+\Vert x\Vert.$$
Consequently $\Vert x+u\Vert=1+\Vert x\Vert$ whenever $x\in Z_\alpha$ for some $\alpha<\mathfrak{p}$. Since $\bigcup_{\alpha<\mathfrak{p}}Z_\alpha$ is dense in $X$, we conclude that $u$ is $L$-orthogonal as desired.\end{proof}

%\newpage
%
%Define, for every $\alpha < \ft$ and $n\in\mathbb N$, the set
%$$C_\alpha:=\bigcap\limits_{n\in  N_\alpha} \overline{co}^{w^*}\{x_j:j>n\}.$$
%
%Then, $\{C_\alpha\}_{\alpha < \ft}$ is a decreasing sequence of weak*-compact sets, so it has nonempty intersection. Let $u \in \bigcap_{\alpha <\ft} C_\alpha$.
%In order to prove that $$\Vert x+u\Vert=1+\Vert x\Vert$$ for every $x \in X$, it is enough to prove that the equality holds for every $x \in \bigcup_{\alpha < \ft} Z_\alpha$, since this set is dense in $X$.
%But if $x \in Z_\alpha$, since $u \in C_\alpha =\bigcap\limits_{n\in  N_\alpha} \overline{co}^{w^*}\{x_j:j>n\}$, it follows from property (ii) and condition (2) of Lemma 0.1 that $\Vert x+u\Vert=1+\Vert x\Vert$.

\section{Limits through selective ultrafilters}\label{sect:selective}

%An ideal on a set $S$ is a proper nonempty family of subsets of $S$ that is closed under taking subsets and under taking finite unions. A filter on $S$ is a proper nonempty family of subsets of $S$ that is closed under taking supersets and under taking finite intersections. Unless otherwise stated, the underlying set will be assumed to be $S=\mathbb{N}$. %An ultrafilter $\mathcal{U}$ is a filter that is not properly contained in any other filter. This is characterized by the fact that for every finite partition of $S$ one of the sets in the partition belongs to $\mathcal{U}$. The trivial or principal ultrafilters are those of the form $\{A : n\in A\}$ for some $n$. The nonprincipal ultrafilters are those that consist only of infinite sets. 

A nonprincipal ultrafilter $\mathcal{U}$ is \emph{selective} (a.k.a. \emph{Ramsey}) if for every partition $\mathbb{N} = \bigcup_n A_n$ into sets $A_n\not\in\mathcal{U}$ there exists a set in the ultrafilter that meets each set of the partition at exactly one point. Selective ultrafilters exist under CH, even under $\mathfrak{p}=\mathfrak{c}$, and in many other models of set theory, cf. \cite[Proposition 10.9, Chapter 25]{halbeisen} and \cite[Section 4.5]{barto}. We will not work with the previous definition but with the following characterization due to Mathias \cite[Theorem 2.12]{mathias}.

\begin{theorem}\label{Mathiastheorem}
	A nonprincipal ultrafilter $\mathcal U$ is selective if and only if the intersection of $\mathcal{U}$ with any dense analytic ideal is nonempty.
\end{theorem}

Let us explain the terminology used. An ideal $I$ is called \textit{dense} (or \emph{tall}) if for every infinite set $A\subset\mathbb{N}$ there exists a further infinite subset $B\subset A$ with $B\in I$. A \emph{Polish space} is a topological space that is homeomorphic to a separable complete metric space. We can view $\mathcal{P}(\mathbb{N})$ as a Polish space through the natural bijection with $\{0,1\}^\mathbb{N}$ that identifies each set $A$ with its characteristic function. A subset $Y$ of a Polish space is called \emph{analytic} if there is a continuous surjection from a Polish space $C$ onto $Y$. We refer to \cite{Kechris} for the basics. The following is an elementary fact (and it holds more generally even for analytic $J$), but we state it and prove it here for the convenience of the reader. 

\begin{fact}\label{idealclosed}
	If $J\subset \mathcal{P}(\mathbb{N})$ is a closed set, then the ideal $I$ generated by $J$,
	$$I = \left\{A\subseteq\mathbb{N} : \exists B_1,\ldots,B_n\in J \ A\subseteq  B_1\cup\cdots\cup B_n\right\},$$
	is analytic. 
\end{fact}

\begin{proof}
	Consider $Z = \mathcal{P}(\mathbb{N})\times \mathbb{N}\times J^\mathbb{N}$, that we endow with the product topology. Countable products of Polish spaces are Polish, so this is a Polish space. The set
	$$C = \{(A,n,B_1,B_2,\ldots)\in Z : A\subseteq B_1\cup\cdots\cup B_n\}$$
	 is closed in $Z$. Indeed, if $(A,n,B_1,B_2,\ldots)\in Z\setminus C$ then there is $x\in A\setminus  (B_1\cup\cdots\cup B_n)$ and the set $U=\{(A',n',B_1',B_2',\ldots): x\in A', ~n'=n, ~x\notin B_i' \mbox{ for every }i\leq n \}$ is an open neighborhood of $(A,n,B_1,B_2,\ldots)$ with $Z\cap C= \emptyset$. Since $C$ is closed in $Z$, it follows that $C$ is also Polish. The projection onto the first coordinate provides a continuous surjection from $C$ onto $I$.
\end{proof}

Remember that, given a sequence $\{x_n\}$ in a topological space and a filter $\mathcal{F}$, the point $x$ is the limit of $\{x_n\}$ through the filter $\mathcal{F}$ if for every $W$ neighborhood of $x$, we have $\{n : x_n\in W\}\in\mathcal{F}$. In such a case, we write $x = \lim_\mathcal{F} x_n$. 
The closure of a sequence can be characterized as the set of all limits of $\{x_n\}$ through ultrafilters, and the cluster points are the limits through nonprincipal ultrafilters. In a Hausdorff space limits through filters are unique and in a compact space limits through ultrafilters always exist.

\begin{theorem}\label{theo:selectpositive}
Let $\{x_n\}$ be an $L$-orthogonal sequence in a Banach space $X$, and let $\mathcal U$ be a selective ultrafilter. Then the $w^*$-limit of $\{x_n\}$ through $\mathcal{U}$ is an $L$-orthogonal element.
\end{theorem}

\begin{proof}
We call $x^{**} := \lim_\mathcal{U}x_n$. Let $z\in X$ and let us prove that $\Vert z+x^{**}\Vert=1+\|z\|$. We fix a sequence $\{\varepsilon_n\}$ of positive numbers converging to $0$. Define $J\subseteq \mathcal P(\mathbb N)$ as the family of all subsets $A\subseteq \mathbb{N}$ such that 
$$(1-\varepsilon_n)(1+\Vert z\Vert)\leq \left\Vert z+w\right\Vert$$
whenever $w$ is a convex linear combination of vectors $\{x_m : m\in B\}$ where $B\subseteq A$ and $|\{m\in A : m<\min(B)\}|\geq n$. 

We claim that $J$ is a closed subset of $\mathcal P(\mathbb N)$. We check that the complement of $J$ is open. The first observation is that, in the definition of $J$, it does not matter if we consider only finite convex combinations or possibly infinite convex combinations. So if $A\not\in J$ it is because we can find a finite convex combination $w = \sum_{j\in B} \alpha_j x_j$ with $B\subseteq A$, $|\{m\in A : m<\min(B)\}|\geq n$ and  $(1-\varepsilon_n)(1+\Vert  z\Vert)> \left\Vert  z+w\right\Vert$. Then $$V = \{A'\subset \mathbb{N} : A'\cap\{1,,\ldots,\max(B)\} = A \cap\{1,\ldots,\max(B)\}\}$$
is a neighborhood of $A$ in $\mathcal{P}(\mathbb{N})$ that is disjoint from $J$ because all sets $A'\in V$ fail to be in $J$ for the same reason as $A$.

By Fact~\ref{idealclosed}, the ideal $I$ generated by $J$ is analytic. Moreover, $I$ is dense because, given any infinite set $A\subseteq\mathbb N$, the sequence $\{x_n\}_{n\in A}$ is $L$-orthogonal, so an application of Lemma \ref{lemma:previmaure}, taking as $Z=span\{z\}$ yields an infinite subset $B\subseteq A$ such that $B\in J$.

Consequently, by Theorem~\ref{Mathiastheorem}, there exists $A\in \mathcal U\cap I$. Since $A\in I$,  $A$ is contained in a finite union of elements of $J$. Since $A\in\mathcal{U}$ and $\mathcal{U}$ is an ultrafilter, one of the sets in that union satisfies $B\in \mathcal U\cap J$. As $B\in\mathcal U$, $x^{**} = \lim_\mathcal{U}x_n$ is a $w^*$-cluster point of $\{x_n\}_{n\in B}$. 
Finally, the fact that $B\in J$ implies that the sequence $\{x_n:n\in B\}$ satisfies the assumptions of Lemma \ref{lemma:sepamaurey} for $Z=span\{z\}$ by a homogeneity argument (see e.g. \cite[Lemma 2.2]{kad}). Hence $\Vert z+x^{**}\Vert=1+\Vert z\Vert$.\end{proof}

One remark is that, since selective ultrafilters exist under the assumption that $\mathfrak{p} = \mathfrak{c}$, for Banach spaces of density $\mathfrak{c}$ Theorem \ref{theo:selectpositive} is more informative than Theorem \ref{theo:octahedral}.

We devote the rest of this section to discuss a question by Kadets, Shepelska and Werner from \cite{ksw}. They proved \cite[Theorem 4.3]{ksw} that every Banach space containing $\ell_1$ admits an equivalent renorming with an $L$-orthogonal sequence. On the other hand, every Banach space containing $\ell_1$ admits an equivalent renorming with $L$-orthogonals \cite[Theorem II.4]{god}. It was highlighted at the very end of \cite{ksw} that the relation between these results was unclear. 
Notice that Theorem \ref{theo:selectpositive} states that, under the existence of selective ultafilters (and, in particular, under CH), the existence of an $L$-orthogonal sequence implies the existence of  $L$-orthogonals in the bidual and, consequently, the renorming of \cite[Theorem 4.3]{ksw} produces an $L$-orthogonal element.

We will show that, in the reverse way, the renorming described in \cite[Theorem II.4]{god} has an $L$-orthogonal sequence. Let $X$ be a Banach space containing a subspace $Y$ which is isomorphic to $\ell_1$. We assume, up to an equivalent renorming, that $Y$ is isometrically isomorphic to $\ell_1$, and let $\{e_n\}$ be the canonical $\ell_1$ basis. Let $i:Y\longrightarrow X$ be the inclusion mapping and $Q:X^*\longrightarrow Y^*=\ell_\infty$ the canonical quotient map ($Q=i^*$). Define
$$\mathcal K:=\{K\subseteq B_{X^*}:K\mbox{ is } w^*\mbox{-compact and } Q(K)=\ext{B_{Y^*}}\}.$$
Pick an element $K_0\in\mathcal K$ which is minimal. Define
$$K_1:=(K_0\cup -K_0)+\{y\in Y^\perp: \Vert y\Vert\leq 2\},$$
and finally
$$B:=\overline{\co}^{w^*}(K_1).$$
In \cite[Theorem II.4]{god} it is proved that $B$ is the unit ball of an equivalent dual norm on $X^*$ whose predual norm $|||\cdot|||$ satisfies that $(X,|||\cdot|||)$ has $L$-orthogonals.

We aim to prove that $|||x+i(e_n)|||\rightarrow 1+|||x|||$ holds for every $x\in X$. So pick $x\in X$ and $\varepsilon>0$, and let us find $m\in\mathbb N$ so that $|||x+i(e_n)|||>1+|||x|||-\varepsilon$ holds for every $n\geq m$.
Since the dual ball of $(X,|||\cdot|||)$ is $B$, we get that $K_1$ is norming for $X$. Hence, we can find $h\in K_1$ such that $h(x)>|||x|||-\varepsilon$. By the form of $K_1$ we can write $h=f+g$ for $f\in K_0\cup -K_0$ (we assume with no loss of generality that $f\in K_0$, since $-K_0$ is also minimal in $\mathcal K$) and $g\in Y^\perp$ with $\|g\| \leq 2$. Define
$$V:=\{\phi\in K_0: \phi(x)>|||x|||-g(x)-\varepsilon\},$$
which is a non-empty ($f\in V$) $w^*$-open subset of $K_0$. Notice that $Q(V)$ has non-empty $w^*$-interior in $\ext{B_{Y^*}}$ by the minimality of $K_0$, since otherwise $K_0\setminus V$ would belong to $\mathcal K$, contradicting the minimality of $K_0$. So $Q(V)$ contains some non-emtpy $w^*$-open subset of $\ext{B_{Y^*}}$, say $U$. The $w^*$-topology of $B_{Y^\ast}=B_{\ell_\infty} = [-1,1]^\mathbb{N}$ coincides with the product topology, so we can find $m\in\mathbb N$ with the property that if $u_1\in U, u_2\in \ext{B_{Y^*}}$ and $u_1(e_i)=u_2(e_i)$ for $1\leq i\leq m$ then $u_2\in U$. With this property in mind, define $u\in Y^*$ by the equation
$$u(e_i):=\left\{\begin{array}{cc}
Q(f)(e_i) & \mbox{if }i\leq m,\\
1 & \mbox{otherwise}
\end{array} \right.$$
It is clear that $u\in \ext{B_{Y^*}}$ and, by the property previously described of $U$, we get that $u\in U\subseteq Q(V)$, so there exists $f'\in V$ such that $Q(f')=u$. This means that
$$f'(i(e_n))=Q(f')(e_n)=u(e_n)=1$$
holds for every $n>m$. Moreover, since $f'\in V$ we get that $f'(x)+g(x)>|||x|||-\varepsilon$. Hence, for $n>m$, we get
\[
\begin{split}
|||x+i(e_n)|||& \geq (f'+g)(x+i(e_n))>|||x|||-\varepsilon+f'(i(e_n))+g(i(e_n))=1+|||x|||-\varepsilon
\end{split}
\]
since $g(y)=0$ for every $y\in Y$. Since $\varepsilon>0$ was arbitrary we conclude the desired result.
%\end{remark}

\section{A Lebesgue filter}\label{section:Lebesgue}

Let $(A_n)$ be a partition of $\nat$ with $\lim_n |A_n| =\infty$ (notice that we allow the sets $A_n$ to be infinite). Set $\I$ the ideal
\begin{equation}\label{eq:ideallebesgue}
	\mathcal{I}= \set{B\subseteq \mathbb{N}: \sup_n|A_n \cap B| < \infty }.
\end{equation}

Following the notation of \cite{KadetsLeonov}, a filter $\mathcal F$ over $\mathbb N$ is said to be \textit{Lebesgue} if the Lebesgue dominated converge theorem holds for convergence through $\mathcal{F}$: for every (finite) measure space $(\Omega,\Sigma,\mu)$ and for every sequence $f_n$ of measurable functions on $\Omega$ with  $\vert f_n\vert$ dominated by a fixed integrable function $g\in L_1(\Omega,\Sigma,\mu)$, if $f_n$ pointwise converges to $0$ through the filter $\mathcal{F}$, then $\lim_\mathcal{F}\int_{\Omega} f_n\ d\mu =  0$.

It will be convenient here to refer these notions to ideals instead of filters. So an ideal $\I$ will be called Lebegue if so is its dual filter $\mathcal F:=\{\mathbb N\setminus A: A\in I\}$. Similarly, convergence through the ideal $\mathcal{I}$ means convergence through its dual filter $\mathcal{F}$. The aim of this section if to show that the ideal $\mathcal{I}$ defined above is Lebesgue.

In the following lemma, given a set $S$ and $n\in\N$, we denote by $[S]^n:=\{A \subseteq S: |A|=n\}$. We would like to thank Grzegorz Plebanek for a substantial simplification of the proof of the following combinatorial lemma.

\begin{lemma}
	\label{LEMAUXCombinatorics}
	Let $(\Omega, \Sigma, \mu)$ be a finite measure space and $\varepsilon>\varepsilon'>0$. Then, for every $N \in \nat$ there is a number $R \in \nat$ such that 
	if $\set{\Omega_n}_{n\leq R}$ is a family of measurable sets with $\mu(\Omega_n)> \varepsilon$ for every $n \leq R$ then  
	$$\mu \left(\bigcup_{S \in [\set{1,\ldots,R}]^N}\bigcap_{j \in S} \Omega_j \right) > \varepsilon'.$$   
\end{lemma}
\begin{proof}
	Without loss of generality, we suppose that $\mu$ is a positive measure. We claim that it is enough to take $R$ big enough so that $\frac{R\varepsilon-N\mu(\Omega)}{R-N}>\varepsilon'$. 
	Fix $\set{\Omega_n}_{n\leq R}$ a family of measurable sets with $\mu(\Omega_n)> \varepsilon$ and consider $g=\sum_{n\leq R} \chi_{\Omega_n}$  the sum of the characteristic functions of the sets $\Omega_n$. Set $$ C:=\bigcup_{S \in [\set{1,\ldots,R}]^N}\bigcap_{j \in S} \Omega_j .$$
	Notice that $g(x)\geq N$ if and only if 
	$x \in C$.
	Then, 
	$$R\varepsilon < \sum_{n\leq R}\mu(\Omega_n)= \int_{\Omega} gd\mu = \int_C gd\mu + \int_{\Omega\setminus C} g d\mu \leq  \int_C gd\mu + \mu(\Omega\setminus C) N \leq$$
	$$ \leq R \mu(C)+(\mu(\Omega)-\mu(C))N = (R-N)\mu(C)+N\mu(\Omega).$$
	Thus, $\mu(C)>\frac{R\varepsilon-N\mu(\Omega)}{R-N}>\varepsilon'$ as desired.
\end{proof}

\begin{proposition}
	\label{PropLebesgueIdeal}
	The ideal $\I$ defined at \eqref{eq:ideallebesgue} is Lebesgue.
\end{proposition}
\begin{proof}
	By \cite[Theorem 2.2]{KadetsLeonov}, it is enough to check that if a sequence of measurable functions $\chi_{\Omega_n}$ pointwise $\I$-converges to zero then $\lim_{\mathcal{I}}\mu(\Omega_n)= 0$.  
	Suppose by contradiction that $\mu(\Omega_n)$ does not $\I$-converge to zero. This means that there is $\varepsilon>0$ such that
	$\{n \in \nat: \mu(\Omega_n)> \varepsilon\} \notin \I$. By definition of $\I$, this means that $\sup_m |\{n \in A_m:\mu(\Omega_n)> \varepsilon \}| = \infty$.
	Fix $0<\varepsilon'<\varepsilon$. Then, by Lemma \ref{LEMAUXCombinatorics}, for every $j$ there is $m_j \in \nat$ big enough such that  $$\mu \left(\bigcup_{S\in [A_{m_j}]^j}\bigcap_{n \in S} \Omega_n \right) > \varepsilon'.$$  
	
	Set $\Omega_j'=\bigcup_{S\in [A_{m_j}]^j}\bigcap_{n \in S} \Omega_n $.
	Then, $\left(\bigcup_{j>n} \Omega_j'\right)_{n \in \nat}$ is a decreasing sequence of sets with measure $>\varepsilon'$.
	This implies that $\mu(\bigcap_{n \in \nat} \bigcup_{j>n} \Omega_j') \geq \varepsilon'$. In particular, there exists $x \in\bigcap_{n \in \nat} \bigcup_{j>n} \Omega_j'$. Notice that this implies that the set $\{j \in \nat: x\in \Omega_j'\}$ is infinite. For every $j$ in this set we have $|\{n\in A_{m_j}: x\in \Omega_n\}|\geq j$.
	In particular, this implies that the set $\{n \in \nat: x\in \Omega_n\} \notin \I$.
	But, since $\chi_{\Omega_n}$ pointwise $\I$-converges to zero, we have that $\{n \in \nat: \chi_{\Omega_n}(x)=1\}=\{n \in \nat: x\in \Omega_n\}\in \I$, which is a contradiction.
\end{proof}

We finish this section with a general fact about Lebesgue filters, cf. \cite{KadetsLeonov2}.
\begin{proposition}\label{pointwiseweakLebesgue}
	If $\mathcal{F}$ is a Lebesgue filter and $\{f_n\}$ is a bounded sequence of $C(K)$ that pointwise converges to $f\in C(K)$ through the filter $\mathcal{F}$, then the sequence weakly converges to $f$ through the filter $\mathcal{F}$.
\end{proposition}

\begin{proof}
	First of all, it follows easily from the definitions that $\lim_\mathcal{F}f_n = f$ in the pointwise topology if and only if $\lim_\mathcal{F}f_n(x) = f(x)$ for every $x\in K$, and that  $\lim_\mathcal{F}f_n = f$ in the weak topology if and only if $\lim_\mathcal{F}y^\ast(f_n) = y^\ast(f)$ for every $y^\ast\in C(K)^*$.
	By Riesz's theorem, every element of $C(K)^\ast$ is given by integration with respect to a measure $\nu$, so the conclusion follows from the definition of a Lebesgue filter.
\end{proof}

\section{Cluster points of $L$-orthogonal sequences that are not $L$-orthogonal}\label{section:counterexamples}

We begin with an example of an $L$-orthogonal sequence that weakly accumulates at $0$. Far from being a sophisticated space, this may occur in the canonical octahedral space $\ell_1$.

\begin{example}\label{example:ejeml1}
There exists an orthogonal sequence $\{x_n\}$ in $\ell_1$ so that $0\in \overline{\{x_n\}}^{w}$. In particular, there exists an ultrafilter $\mathcal U$ so that $\lim_\mathcal U x_n=0$.
\end{example}

\begin{proof}
Define, for every $n\in\mathbb N$, $A_n$ to be the set of elements $x\in B_{\ell_1}$ such that $\Vert x\Vert \geq 1-\frac{1}{2^n}$, $x_i\in\{\frac{z}{4^n}: z\in\mathbb Z\}$ if $2^n<i\leq 2^{n+1}$ and $x_i=0$ otherwise. Notice that $A_n$ is finite for every $n\in\mathbb N$, so $A:=\bigcup\limits_{n\in\mathbb N}A_n$ is a countable subset of $S_{\ell_1}$.

Note that if we consider $A$ as a sequence in $S_{\ell_1}$ then $A$ is an orthogonal sequence. This easily follows from the density of finitely-supported sequences in $\ell_1$ and from the fact that, for every $k\in\mathbb N$, the set $\{x\in A: \supp(x)\cap \{1,\ldots, k\}\neq \emptyset\}$ is finite for every $k\in\mathbb N$.

Let us prove that $0$ is a weak cluster point of $A$. To this end, pick a basic weak neighbourhood $O:=\{x\in B_{\ell_1}: \vert x_i^*(x)\vert<\varepsilon, \mbox{ for every } 1\leq i\leq p\}$ of $0$, with $x^*_1,\ldots, x^*_p\in S_{\ell_\infty}$ and $\varepsilon>0$, and let us prove that $O\cap A\neq\emptyset$. For this, pick $n$ large enough so that $2^n>p$ and $\frac{1}{2^n}<\varepsilon$. By a dimension argument we can find an element $y\in S_{\ell_1}$ which is supported on the coordinates $\{2^n+1,\ldots, 2^{n+1}\}$ and such that $x_i^*(y)=0$ holds for $1\leq i\leq p$. Let $x\in B_{\ell_1}$ be supported on the coordinates $\{2^n+1,\ldots, 2^{n+1}\}$ such that $x_i=\frac{k}{4^n}$ and $\vert x_i-y_i\vert<\frac{1}{4^n}$. Note hence that $\Vert x-y\Vert<\frac{1}{2^n}$. This implies, on the one hand, that $\Vert x\Vert>1-\frac{1}{2^n}$, from where $x\in A_n\subseteq A$. On the other hand, $ x_i^*(x)=0$ for every $i\leq p$, which proves that $x\in O$ as desired.
\end{proof}

Our next step is to get other examples of that sort where we have a better understanding of what kind of ultrafilters are giving non $L$-orthogonal cluster points.

%
%%\textbf{Control FVV1. 12/12/2017 }.\\
%
%First, we generalize the ideal of statistical convergence as follows.
%For any infinite set $A$, we define the density function 
%$$d_A(B)=\limsup_n \frac{\left|B\cap A \cap \{1,\ldots,n\} \right|}{|A\cap \{1,\ldots,n\}|} \mbox{ for every }B \subset \nat.$$ 
%Let $(A_n)$ be a partition of $\nat$. Set $\I$ the ideal
%$$\mathcal{I}= \set{B\subset\mathbb{N}: d_{A_n}(B) = 0 \mbox{ for every }n \in \nat}.$$
%
%\begin{lem}
%\label{LemmaLebesgueFilter}
%The filter $\F$ associated to the ideal $\I$ is Lebesgue.
%\end{lem}
%\begin{proof}
%By \cite[Proposition 3.1]{KadetsLeonov}, it is enough to check that $\F$ is generated by a summability matrix. Let $\phi=(\phi_1,\phi_2):\nat \rightarrow \nat \times \nat$ be a bijection.
%Let $\varphi=(\varphi_{i,j})$ be the matrix given by the formula $\varphi_{i,j}=\frac{1}{|A_{\phi_1(i)}\cap \{1,\ldots,\phi_2(i)\}|}$ if $j \in |A_{\phi_1(i)}\cap \{1,\ldots,\phi_2(i)\}|$ and zero otherwise.
%
%Then, 
%$$d_{\varphi}(B)= \lim_{i} \sum_{j=1}^{\infty} \varphi_{i,j} \chi_B(j) =  \lim_{i}  \frac{1}{|A_{\phi_1(i)}\cap \{1,\ldots,\phi_2(i)\}|} |B\cap A_{\phi_1(i)}\cap \{1,\ldots,\phi_2(i)\}| $$
%and, bearing in mind that $\phi$ is a bijection, if the previous limit exists then it is $d_\varphi(B)=d_{A_n}(B)$ for every $n \in \nat$.
%Thus, $\I$ is the ideal of $\varphi$-null sets.
%Therefore, to finish the proof it is enough to check that $\varphi$ is a summability matrix, i.e.~that it satisfies conditions (1) to (4) in \cite[pg 10]{KadetsLeonov}, but this is immediate.
%\end{proof}
%

\begin{proposition}
	\label{PropOrthogonalSeqI}
	Suppose that the ideal $\I$ defined at \eqref{eq:ideallebesgue} is determined by a partition $(A_n)$ of $\nat$ into finite sets. Then, there exists an $L$-orthogonal sequence $(e_n)$ in $C(\{0,1\}^\mathbb{N})$ that $\I$-converges to the constant function $1$ in the weak topology. In particular, $\lim_\mathcal{U} e_n$ is the constant function $1$ for every ultrafilter $\mathcal{U}$ disjoint from $\mathcal{I}$.
\end{proposition}

\begin{proof}

%We write $2=\{0,1\}$, $\mathbb{N} = \{1,2,3,\ldots\}$. 
Let	$$ K = \left\{ x = (x_i)\in \{-1,1\}^\mathbb{N} :  \left|\{i\in A_m : x_i=-1\}\right| \leq 1 \mbox{ for every }m\in \N \right\}.$$ Notice that $K$ is a compact subset of $\{-1,1\}^\mathbb{N}$.

{Claim 1.} For every nonempty open subset $V\subset K$ there exists $n_V\in \mathbb{N}$ such that 
$\{x\in V  : x_{n}=-1\}\neq \emptyset$ and  $\{x\in V  : x_{n}=1\}\neq \emptyset$ for every $n > n_V$.

{\textit{Proof of Claim 1.}} We can suppose that $V$ is a nonempty basic open set given by
$$V = \set{x \in K : x_i = y_i \text{ for all } i\leq r  }$$
for some $r$ and some $y\in K$ that can be taken so that $y_i=1$ for $i>r$. Then it is enough to take $n_V = \max(\bigcup_{i\leq r, i \in A_j} A_j )$.

From Claim 1 it follows that $K$ does not have isolated points. So $K$ is a metrizable zero-dimensional perfect compactum, hence homeomorphic to $\{0,1\}^\mathbb{N}$ by Brower's theorem, cf. \cite[Theorem 7.4]{Kechris}. We work in $C(K)$ instead of $C(\{0,1\}^\mathbb{N})$.
Define $e_n\in C(K)$ as the $n$-th coordinate function on $K$.  Let us first check that this is an $L$-orthogonal sequence. Fix a function $f\in C(K)$ and a point $z\in K$ where $|f|$ attains its maximum. Write $f(z) = \eta \|f\|$ with $\eta\in\{-1,1\}$. Given $\varepsilon>0$, we can find an open neighborhood $V$ of $z$ such that $|f(z)-f(x)|<\varepsilon$ for all $x\in V$. By Claim 1, for each $n>n_V$ we can find $x\in V$ such that $x_n = \eta$. Therefore
$$ \|f\| + 1 - \varepsilon = |\eta \|f\| + \eta| -\varepsilon \leq  |f(x)+x_n| \leq \|f+e_n\| \leq \|f\| + 1$$
for $n>n_V$. This shows that the sequence $\{e_n\}$ is $L$-orthogonal.

Finally, notice that for every $x\in K$, 
$$\set{n\in\mathbb{N} : e_n(x) \neq 1} = \set{n\in\mathbb{N} : x_{n} = -1} \in \mathcal{I}.$$

Therefore, $(e_n)$ pointwise converges to $1$ through the ideal $\mathcal{I}$. 
Since $\I$ is Lebesgue by Proposition \ref{PropLebesgueIdeal}, using Propostion \ref{pointwiseweakLebesgue} we conclude that $(e_n)$ converges to the constant function $1$ in the weak topology through the ideal $\mathcal{I}$.
\end{proof}

A nonprincipal ultrafilter $\mathcal{U}$ is a \emph{$Q$-point} if for every partition of $\mathbb{N}$ into nonempty finite sets $\mathbb{N} = \bigcup_n A_n$ there exists a set in the ultrafilter that meets each set of the partition at exactly one point. This is equivalent to saying that there is a set $B\in\mathcal{U}$ such that $\sup_n |B\cap A_n| <+\infty$ (see the proof of Corollary \ref{CorOrthogonalNoQPoint}).  It is clear that every selective ultrafilter is a $Q$-point. There are models of set theory where no $Q$-points exist \cite{miller}.

\begin{corollary}
\label{CorOrthogonalNoQPoint}
Suppose that $\U$ is an ultrafilter which is not a $Q$-point. Then, there is an $L$-orthogonal sequence $(e_n)$ in $C(\{0,1\}^\nat)$ that converges through $\mathcal{U}$ to the constant function $1$ in the weak topology.
\end{corollary}
\begin{proof}
By definition, if $\U$ is not a $Q$-point, then there is a partition $(A_n)$ of $\nat$ into nonempty finite sets such that if $|B\cap A_n|\leq 1$ for every $n$ then $B\notin \U$. Notice that if  $\sup_n |B\cap A_n| <\infty$, then we also get $B\not\in\U$. This is because $B$ can be written as a finite disjoint union of sets $B_j$ for which  $|B_j\cap A_n|\leq 1$ for every $n$. Thus, if $\I$ is the ideal associated to the sets $(A_n)$ as in \eqref{eq:ideallebesgue}, then $\U$ and $\I$ are disjoint. The conclusion follows from Proposition \ref{PropOrthogonalSeqI}.
\end{proof}

Now, we get a single nonseparable space where all failures happen simultaneously.

\begin{theorem}\label{theo:sucesinLorto}
	There is an $L$-orthogonal sequence in a Banach space such that none of whose limits through non $Q$-point ultrafilters is $L$-orthogonal. So if no Q-points exist, there is a Banach space with an $L$-orthogonal sequence  with no $L$-orthogonal elements in its $w^*$-closure.
\end{theorem}

\begin{proof} Let $\mathfrak{F}$ be a family of norm-one operators $T:\ell_1 \rightarrow X_T$ such that $(Te_n)$ is an $L$-orthogonal sequence in $X_T$ for every $T \in \mathfrak{F}$. Consider $X_\mathfrak{F}=\ell_1 \oplus c_0(\mathfrak{F})$ endowed with the norm given by
$$ \|x+y\|= \sup_{T \in \mathfrak{F}} \{ \|x\|_1, \|Tx+y_T1_T\|\} $$ 
for every $x \in \ell_1$ and $y= (y_T)_{T \in \F} \in c_0(\mathfrak{F})$. The first key observation is that the copy of the canonical basis $(e_n)$ of $\ell_1$ is an $L$-orthogonal sequence in $X_\mathfrak{F}$.
Indeed, take any $x+y \in X_\mathfrak{F}$ with $x \in \ell_1$ and $y= (y_T)_{T \in \mathfrak{F}} \in c_0(\mathfrak{F})$. 
If $\|x+y\|=\|x\|_1$, then it is immediate that $\lim_n \|e_n+x+y\|=1+\|x\|_1=1+\|x+y\|$, since $e_n$ is $L$-orthogonal in $\ell_1$.
Otherwise %for any $\varepsilon>0$ we can find $T$ with 
$\|x+y\|= \sup_{T \in \mathfrak{F}}\|Tx+y_T1_T\|$.
Then, since $(Te_n)$ is an $L$-orthogonal sequence in $X_T$ for every $T\in \mathfrak{F}$, we have that
$$\|e_n+x+y\| \geq \|Te_n+Tx+y_T1_T\|\longrightarrow 1+\|Tx+y_T1_T\|.$$ 
Since the last inequality holds for every $T\in \mathfrak{F}$ and $\sup_{T \in \mathfrak{F}}1+\|Tx+y_T1_T\|=1+\|x+y\|$, we conclude that $(e_n)$ is an $L$-orthogonal sequence in $X_\mathfrak{F}$.

According to Corollary \ref{CorOrthogonalNoQPoint}, for each ultrafilter $\mathcal{U}$ that is not a $Q$-point we can put an operator $T$ in our family $\mathfrak{F}$ so that $1_T := w-\lim_\mathcal{U}Te_n$ is an element of $X_T$ of norm 1. It remains to prove that in such a situation $e_T^{**}=w^*\textendash \lim_{\U} e_n\in X_\mathfrak{F}^{**}$  is not an $L$-orthogonal element.
For this, it is enough to show that $\|e_T^{**}-\delta_T\|\leq 1$, where $\delta_T \in c_0(\mathfrak{F})$ stands for the vector whose coordinates are all zero except the coordinate $T$ which is one. Let $x^*+y^* \in X_\mathfrak{F}^*=\ell_\infty \oplus \ell_1(\mathfrak{F})$ be such that $(x^*+y^*)(e_T^{**}-\delta_T)=x^*(e_T^{**})-y^*(\delta_T)>1$.
If we prove that $\|x^*+y^*\|>1$, then we will have that $\|e_T^{**}-\delta_T\|\leq 1$.

Since  $x^*(e_T^{**})>1+y^*(\delta_T)$ and $e_T^{**}=w^*\textendash \lim_{\U} e_n$, we get that
$$ S=\{ n \in \nat: x^*(e_n)>1+y^*(\delta_T)\} \in \U.$$
On the other hand $w-\lim_\U Te_n - 1_T = 0$, so $0\in\overline{\{Te_n - 1_T : n\in S\}}^w$. By Mazur's lemma,
$$ 0 \in \overline{co}^{\|\cdot \|} \{Te_n-1_T: n \in S\}.$$
Let $x \in co\{e_n: n \in S\}$ be an element with $\|Tx-1_T\| \leq 1$.
Then, $\|x-\delta_T\| = 1$ but
$$ (x^*+y^*)(x-\delta_T)>1+y^*(\delta_T)-y^*(\delta_T)=1,$$

so $\|x^*+y^*\|>1$, and this finishes the proof.
\end{proof}

It is possible to take the construction above a little further. For every ultrafilter $\mathcal{U}$ that is not a $Q$-point and every choice of signs $(\sigma_n)\in \{-1,1\}^\mathbb{N}$, we can put an operator $T:\ell_1\To X_T$ in our family $\mathfrak{F}$ so that  $w-\lim_\mathcal{U}\sigma_nTe_n$ is an element $1_T$ of $X_T$ of norm 1. If there are no $Q$-points, we will get that no element of the $w^*$-closure of $\{e_n, -e_n : n\in\mathbb{N}\}$ is $L$-orthogonal. Let us recall that we do not know whether a negative answer to Question \ref{question:general} is consistent. It is natural to conjecture that a space of this sort might be a counterexample, but we do not know how to discard $L$-orthogonality for an arbitrary element of the bidual. Notice that $X_\mathfrak{F} ^{**}= \ell_\infty^* \oplus \ell_\infty(\mathfrak{F})$.
It can be easily checked that if an element $x^{**}+y^{**} \in \ell_\infty^* \oplus \ell_\infty(\mathfrak{F})$ is $L$-orthogonal, then $y^{**}=0$ and $x^{**} \in \overline{co}^{w^*}\{\sigma e_n: n \in \nat,~\sigma \in \{-1,1\}\}$.
No extreme point of the ball $B_{X_\mathfrak{F}^{**}}$ would be $L$-orthogonal. Does this property imply that $X_\mathfrak{F}^{**}$ does not have $L$-orthogonal elements?

\section{Abundance of $L$-orthogonals}\label{section:abundLorto}

So far, we only worried about the existence or not of $L$-orthogonal elements, but now we will examine how large the set of $L$-orthogonal elements might be in a separable Banach space. One result in this direction is that the set of $L$-orthogonal elements is $w^*$-dense in $B_{X^*}$ if and only if $X$ has the Daugavet property \cite[Theorem 3.2]{rueda}. It is convenient now to remove normalization, so for a Banach space $X$ consider the set
\begin{eqnarray*}L_X &:=& \{u\in X^{**}: \| x +u \| = \|x\| +\|u\| \text{ for all }x\in X\}%\\ &=& \{\lambda u : u \text{ is an }L\text{-orthogonal element}, \lambda\in\mathbb{R}\}.
\end{eqnarray*}
Note that, with the terminology of the introduction, the set of $L$-orthogonal elements to $X$ is nothing but $L_X\cap S_{X^{**}}$. Moreover, notice that if $u\in L_X\setminus \{0\}$ then $\frac{u}{\Vert u\Vert}$ is an $L$-orthogonal by Lemma \ref{lemma:abal}.

Our aim is to prove, working with injective tensor products, that when $X$ is a separable octahedral space, $L_X$ contains infinite-dimensional Banach spaces.

The following proposition strengthens \cite[Proposition 3.1]{loru}. Let us explain the terminology. $L(X,Y)$ (respectively $K(X,Y)$) is the space of all linear and bounded (respectively linear and compact) operators from $X$ to $Y$. $X\pten Y$ and $X\iten Y$ are the projective and injective tensor product of $X$ and $Y$, respectively. See \cite{rya} for a detailed treatment of the tensor product theory and approximation properties.

\begin{proposition}\label{prop:previocontraeje}
Let $X$ be a uniformly smooth Banach space and let $Y$ be a Banach space such that either $X^*$ or $Y^*$ has the approximation property. Let $T\in  L(X^*,Y^{**}) = (X\iten Y)^{**}=(X^*\pten Y^*)^*$ be such that $\Vert T\Vert=1$ and 
$$\Vert T+S\Vert=2$$
 for every norm-one element $S\in K(X^*,Y) = X\iten Y$. Then $T$ is an isometry and $T(x^*)$ is $L$-orthogonal to $Y$ for every $x^*\in X^*$.
\end{proposition}

\begin{proof}
Note that since $X^*$ has the Radon-Nikodym property by reflexivity and since $X^*$ or $Y^*$ has the approximation property, then $(X\iten Y)^*=X^*\pten Y^*$ \cite[Theorem 5.33]{rya}. Consequently $L(X^*,Y^{**}) =(X^*\pten Y^*)^*= (X\iten Y)^{**}$.

We fix $x^*\in S_{X^*}$ and $y\in Y$ and we prove that $\Vert T(x^*)+y\Vert=1+\|y\|$. This will be enough. Since $X$ is reflexive, we can find $x\in S_X$ such that $x^*(x)=1$. Since $\Vert T+x\otimes y\Vert=1+\|y\|$ we can find, for every $n\in\mathbb N$, an element $x_n^*\in S_{X^*}$ and $y_n^*\in S_{Y^*}$ such that
$$1+\|y\|-\frac{1}{n}<y_n^*(T(x_n^*))+x_n^*(x)y_n^*(y).$$
This implies that $x_n^*(x)\rightarrow 1$ and $y_n^*(y)\rightarrow \|y\|$. Moreover, evaluation at $x$ gives that $\Vert x_n^*+x^*\Vert\rightarrow 2$ from where we get, since $X^*$ is uniformly convex, that $\Vert x_n^*-x\Vert\rightarrow 0$. Now
\[
\begin{split}
1+\|y\|-\frac{1}{n}<y_n^*(T(x_n^*))+x_n^*(x)y_n^*(y)& \leq y_n^*(T(x^*))+x^*(x)y_n^*(y)\\ & +\Vert y_n^*\circ T+y_n^*(y)x\Vert \Vert x_n^*-x^*\Vert\\
& \leq \Vert T(x^*)+y\Vert+2\Vert x_n^*-x\Vert.
\end{split}
\] 
Since $\Vert x_n^*-x^*\Vert\rightarrow 0$ we deduce that $\Vert T(x^*)+y\Vert=1+\|y\|$.
\end{proof}

\begin{theorem}\label{theo:espaciaorto}
Let $X$ be a separable Banach space whose norm is octahedral and $Y$ be a separable and uniformly smooth Banach space with the metric approximation property which is finitely representable in $\ell_1$. Then $L_X$ contains an isometric copy of $Y^*$.
\end{theorem}

\begin{proof}
Since $Y$ is finitely representable in $X$ and has the metric approximation property, the norm of $K(Y,X)=Y^*\iten X$ is octahedral \cite[Theorem 3.2]{lr}. By \cite[Lemma 9.1]{gk} there exists an element $T\in (Y^*\iten X)^{**}=L(Y^*,X^{**})$ which is an $L$-orthogonal to $Y^*\iten X$. Now Proposition \ref{prop:previocontraeje} implies that $T(y^*)\in L_X$ for every $y^*\in Y^*$ and that $T$ is an isometry, so we are done.
\end{proof}

%\begin{example}
Notice that, in general, the set $L_X$ is not a vector space. Indeed, consider the space of continuous functions $X= C[0,1]$. Remember that every Borel function on $[0,1]$ can be viewed as an element of $X^{**}$, because $X^{*}$ is the space of measures on $[0,1]$ and such a function can act by integration. If we fix $r$ an irrational number, then  
$u:=\chi_{[0,1]\cap \mathbb Q}-\chi_{[0,1]\setminus \mathbb Q}$ and $v:=\chi_{[0,1]\cap (\mathbb Q\cup r)}-\chi_{[0,1]\setminus (\mathbb Q\cup\{r\})}$ are both $L$-orthogonal to $C[0,1]$, but $w=\frac{v-u}{2}=\chi_{\{r\}}$ is not. Indeed, if we take a continuous function $f\in S_{C[0,1]}$ that vanishes on $r$ then $\Vert f+w\Vert=1$.
%\end{example}

An even more surprising is the following result.

\begin{theorem}
Let $X$ and $Y$ be two separable Banach spaces such that $X$ is octahedral and $Y^*$ is uniformly convex and finitely representable in $\ell_1$. Denote
$$Iso(Y^*, X^{**})=\{T:Y^*\longrightarrow X^{**}:T \mbox{ is a linear isometry}\}.$$
Then there exists a linear isometry $T:Y^*\longrightarrow L(Y^*,X^{**})$ such that $T(y^*)\in Iso(Y^*,X^{**})$ holds for every $y^*\in S_{Y^*}$. In particular, $Iso(Y^*,X^{**})$ contains an isometric copy of $S_{Y^*}$.
\end{theorem}

\begin{proof}
The space $X\iten Y$ has an octahedral norm \cite[Theorem 3.2]{lr}. Since $Y^*$ is finitely representable in $\ell_1$, again $Y\iten Y\iten X$ is a separable Banach space with octahedral norm by the same result \cite[Theorem 3.2]{lr}. By \cite[Lemma 9.1]{gk} there exists an $L$-orthogonal element $T\in (Y\iten Y\iten X)^{**}$. By Proposition \ref{prop:previocontraeje}, $T$ is a linear isometry $T:Y^*\longrightarrow L(Y^*,X^{**})$ such that $T(y^*)$ is an $L$-orthogonal to $X\iten Y$ for every element $y^*\in S_{Y^*}$. Again by Proposition \ref{prop:previocontraeje} we get that $T(y^*)$ is a linear isometry $T(y^*):Y^*\longrightarrow X^{**}$ whose range consists of $L$-orthogonal vectors to $X$. Proposition \ref{prop:previocontraeje} finishes the proof.
\end{proof}

A kind of converse of Proposition \ref{prop:previocontraeje} is the following result.

\begin{proposition}
Let $X$ be a Banach space. Let $Y$ be a Banach space such that there exists a linear isometry $T:Y^*\longrightarrow X^{**}$ with $T(y^*)\in L_X$ holds for every $y^*\in Y^*$. Then $T$ is $L$-orthogonal to $Y\iten X$.
\end{proposition}

\begin{proof}
Pick an element $S\in Y\iten X=K(Y^*,X)$ such that $\Vert S\Vert=1$ and $\varepsilon>0$. Find $y^*\in S_{Y^*}$ such that $\Vert S(y^*)\Vert>1-\varepsilon$. Now, since $T(y^*)\in L_X$, we get
$$\Vert T+S\Vert\geq \Vert T(y^*)+S(y^*)\Vert=\Vert T(y^*)\Vert +\Vert S(y^*)\Vert>1-\varepsilon+\Vert T(y^*)\Vert=2-\varepsilon.$$
Since $\varepsilon>0$ was arbitrary we are done.
\end{proof}

 According to \cite{hww}, a Banach space $X$ is said to be an \textit{$L$-embedded space} if there exists a subspace $Z$ of $X^{**}$ such that $X^{**}=X\oplus_1 Z$, where this indicates the direct sum with $\ell_1$-norm: $\|x+z\| = \|x\|+\|z\|$. Examples of $L$-embedded Banach spaces are $L_1(\mu)$ spaces, preduals of von Neumann algebras, duals of $M$-embedded spaces or the dual of the disk algebra (see \cite[Example IV.1.1]{hww} for formal definitions and details). 

\begin{corollary}
Let $X$ be a non-reflexive $L$-embedded Banach space, say $X^{**}=X\oplus_1 Z$. If $Y^*$ is an isometric subspace of $Z$, then $X\iten Y$ has an $L$-orthogonal element.
\end{corollary}

\begin{proof}
It follows from the previous result and the fact that $Z$ is exactly the set of those $L$-orthogonals to $X$.
\end{proof}

\end{document}